\documentclass[11pt]{article} 
\usepackage{amsfonts,amsmath,latexsym,amssymb,mathrsfs,amsthm,comment}
\usepackage[colorlinks=true, linkcolor=black, citecolor=black]{hyperref}
\usepackage{caption}
\usepackage{graphicx}
\usepackage{xcolor}
\usepackage{booktabs}

\let\OLDthebibliography\thebibliography
\renewcommand\thebibliography[1]{
  \OLDthebibliography{#1}
  \setlength{\parskip}{1pt}
  \setlength{\itemsep}{0pt plus 0.0ex}
}

\def\numberlikeadb{\global\def\theequation{\thesection.\arabic{equation}}}
\numberlikeadb
\newtheorem{theorem}{Theorem}[section]
\newtheorem{lemma}[theorem]{Lemma}
\newtheorem{corollary}[theorem]{Corollary}

\newtheorem{proposition}[theorem]{Proposition}
\newtheorem{remark}[theorem]{Remark}

\usepackage{color}

\usepackage{lscape}
\usepackage{caption}
\usepackage{multirow}
\allowdisplaybreaks
\begin{document} 
\title{The variance-gamma product distribution
}
\author{Robert E. Gaunt\footnote{Department of Mathematics, The University of Manchester, Oxford Road, Manchester M13 9PL, UK, robert.gaunt@manchester.ac.uk; siqi.li-8@postgrad.manchester.ac.uk; heather.sutcliffe@manchester.ac.uk},\: Siqi L$\mathrm{i}^{*}$  and Heather L. Sutcliff$\mathrm{e}^{*}$}

\date{} 
\maketitle

\vspace{-8mm}

\begin{abstract} We derive the exact probability density function of the product of $N$ independent variance-gamma random variables with zero location parameter. We then apply this formula to derive formulas for the cumulative distribution function and characteristic function, as well as asymptotic approximations for the density, tail probabilities and quantile function. From our general results, we deduce closed-form formulas for the density,  cumulative distribution function and characteristic function of the product of $N$ independent asymmetric Laplace random variables and mixed products of independent Laplace and centred normal random variables.
\end{abstract}



\section{Introduction}

The variance-gamma (VG) distribution with parameters $m > -1/2$, $0\leq|\beta|<\alpha$, $\mu \in \mathbb{R}$, which we denote by $\mathrm{VG}(m,\alpha,\beta,\mu)$, has probability density function (PDF)
\begin{equation}\label{vgpdf} f(x) = \frac{\gamma^{2m+1}}{\sqrt{\pi}(2\alpha)^m \Gamma(m+1/2)} \mathrm{e}^{\beta (x-\mu)}|x-\mu|^{m} K_{m}(\alpha|x-\mu|), \quad x\in \mathbb{R},
\end{equation}
where $\gamma^2=\alpha^2-\beta^2$ and $K_{m} (x) =\int_0^\infty \mathrm{e}^{-x\cosh(t)}\cosh(m t)\,\mathrm{d}t$ is the modified Bessel function of the second kind. 
The VG distribution is also referred to as the Bessel function distribution \cite{m32}, the generalized Laplace distribution \cite[Section 4.1]{kkp01} and the McKay Type II distribution \cite{ha04}. Other parametrisations can also be found in \cite{gaunt vg,kkp01,mcc98}.
Following the seminal works \cite{mcc98,madan}, the VG distribution has been widely used in financial modelling; further application areas are given in the review \cite{vg review} and Chapter 4 of the book \cite{kkp01}. In this paper, we set $\mu=0$. On further setting $\beta=0$, the PDF (\ref{vgpdf}) is symmetric about the origin, and in this paper we will refer to the $\mathrm{VG}(m,\alpha,0,0)$ distribution as the \emph{symmetric variance-gamma} distribution.

Let $X_1,\ldots,X_N$ be independent VG random variables with $X_i\sim \mathrm{VG}(m_i,\alpha_i,\beta_i,0)$, $m_i>-1/2$, $0\leq|\beta_i|<\alpha_i$, $i=1,\ldots,N$. In this paper, we obtain an exact formula for the PDF of the product $Z=\prod_{i=1}^N X_i$, which is expressed as an infinite series in terms of the Meijer $G$-function. Our formula generalises one of \cite{gl24} for the product of two independent VG random variables (the $N=2$ case) and one of \cite{gaunt algebra} for the PDF of the product $Z$ in the case $\beta_1=\cdots=\beta_N=0$ (a product of symmetric VG random variables). We derive our formula using the theory of Mellin transforms, a method which has been used to find exact formulas for the PDFs of products of many important continuous random variables; see, for example, \cite{e48,ks69,l67,s79,st66,product normal,wells}.  

In the symmetric case $\beta_1=\cdots=\beta_N=0$, the PDF of the product $Z$ reduces to a single Meijer $G$-function, whilst when $m_1,\ldots,m_N$ are half-integers the PDF of $Z$ can be represented as a finite summation of Meijer $G$-functions. In these more tractable cases, we apply our formulas for the PDF to deduce closed-form formulas for the cumulative distribution function (CDF) and characteristic function. Returning to the general case $m_i>-1/2$, $0\leq|\beta_i|<\alpha_i$, $i=1,\ldots,N$, we also derive asymptotic approximations for the PDF, tail probabilities and the quantile function. 
We state and prove these results in Section \ref{sec2}. 

The $\mathrm{VG}(1/2,\alpha,\beta,0)$ distribution corresponds to the asymmetric Laplace distribution (with zero location parameter) and the classical symmetric Laplace distribution is obtained from further setting $\beta=0$.
Thus from our general results of Section \ref{sec2} we obtain formulas for the PDF, CDF and characteristic function of the product of $N$ independent asymmetric Laplace random variables. These formulas (which are given in Section \ref{sec3.1}) fill a surprising gap in the literature, as such formulas were previously only available in \cite{gl24,nad07} for the $N=2$ case. 
The product of two independent zero mean normal random variables is also VG distributed \cite{gaunt vg}, and in Section \ref{sec3.2} we deduce new results for the distribution of mixed products of independent Laplace and zero mean normal random variables. 
More generally, the product of two correlated zero mean normal random variables is VG distributed \cite{gaunt prod}, and so we are also able to deduce an exact formula for the PDF of the product of $2N$ jointly correlated zero mean normal random variables with a block diagonal covariance matrix (see Section \ref{sec3.3}). 

\vspace{2mm}

\noindent \emph{Notation.} 
We let $S_N$ denote the power set of $\{1,\ldots,N\}$, that is the set of all subsets of $\{1,\ldots,N\}$. We let $\sigma\in S_N$ and define its complement  $\sigma^{c}=\{1,\ldots,N\}\setminus\sigma$. We further let $S_N^+=\{\sigma\in S_N\,:\,\text{$|\sigma|$ is even}\}$ and $S_N^-=\{\sigma\in S_N\,:\,\text{$|\sigma|$ is odd}\}$.
Also, $\mathrm{sgn}(x)$ denotes the sign function, $\mathrm{sgn}(x)=-1$ for $x<0$, $\mathrm{sgn}(0)=0$, $\mathrm{sgn}(x)=1$ for $x>0$. Lastly, we use the convention that the empty product is set to 1.

In order to simplify formulae, we let $\eta = 1/\prod^N_{i=1} \Gamma(m_i+1/2)$ and $\xi=\prod^N_{i=1} \alpha_i$.
We also let $\gamma_i^2 = \alpha_i^2-\beta_i^2$ and $\lambda_i^{\pm}=\alpha_i\pm\beta_i$, $i=1,\ldots,N$. Finally,  for $\sigma\in S_N$, we let $\omega_\sigma=\prod_{k \in \sigma}\prod_{l \in \sigma^c}(\alpha_k+\beta_k)(\alpha_l-\beta_l)$. In the case $(\alpha_i,\beta_i)=(\alpha,\beta)$ for $i=1,\ldots,N$, we have $\omega_\sigma=(\alpha+\beta)^{|\sigma|}(\alpha-\beta)^{N-|\sigma|}$.

\section{The VG product distribution
}\label{sec2}

\subsection{Probability density function}\label{sec2.1}
We begin by obtaining an exact formula for the PDF of the product $Z = \prod_{i=1}^{N}X_i$. The formula is expressed as an infinite series given in terms of the Meijer $G$-function (defined in Appendix \ref{appa}). 



\begin{theorem}\label{thm1}
	Let $X_1,\ldots,X_N$ be $N$ independent VG distributed random variables, where $X_i \sim \mathrm{VG}(m_i,\alpha_i,\beta_i,0)$ with $m_i> -1/2$ and $0 \leq |\beta_i| < \alpha_i $, for $i=1,\ldots,N$. Let $Z = \prod_{i=1}^{N}X_i$. Then
	\begin{align}
		f_Z(z) & = \frac{\eta}{2^{2N-1}\pi^{N/2}}\bigg\{\prod_{i=1}^{N}\frac{\gamma_i^{2m_i+1}}{\alpha_i^{2m_i}}\bigg\}
		\sum^\infty_{j_1=0} \ldots \sum^\infty_{j_N=0} \frac{(2\beta_{1}/\alpha_1)^{j_1} }{ j_1! } \ldots \frac{(2\beta_N/\alpha_N)^{j_N}}{j_N!} a_{j_1,\ldots,j_N}(z) \nonumber
		\\ & \quad \times G^{2N,0}_{0,2N} \bigg( \frac{\xi^{2}}{4^N}z^2\, \bigg|\, { - \atop   \tfrac{j_1}{2},\ldots, \tfrac{j_N}{2}, m_1+\tfrac{j_1}{2}, \ldots, m_N+\tfrac{j_N}{2} }   \bigg),\quad z\in\mathbb{R}, \label{eq:1}
	\end{align}
	where $a_{j_1,\ldots,j_N}(z)=\sum_{\sigma \in S_N^+ } \prod_{k \in \sigma} (-1)^{j_k}$ for $z>0$ and $a_{j_1,\ldots,j_N}(z)=\sum_{\sigma \in S_N^-} \prod_{k \in \sigma} (-1)^{j_k}$ for $z<0$.
\end{theorem}

\begin{remark}\label{firstrem}
In the case of a product of independent symmetric VG random variables ($\beta_1=\cdots=\beta_N=0$), the PDF $f_Z(z)$ is symmetric about the origin and reduces to a single $G$-function:
	\begin{equation}
		f_Z(z) = \frac{\xi \eta}{2^N \pi^{N/2}} G^{2N,0}_{0,2N} \bigg( \frac{\xi^{2}}{4^N}z^2\, \bigg|\, \begin{aligned} - \atop  0, \ldots,0, \, m_1,\ldots, m_N  \end{aligned}   \bigg),\quad z\in\mathbb{R},\label{simp}
	\end{equation}
    a formula that had earlier been derived by \cite{gaunt algebra}.
	When $N-1$ of the skewness parameters are equal to zero, the PDF (\ref{eq:1}) simplifies to a single infinite series. Without loss of generality, set $\beta_2=\cdots=\beta_N=0$. Then, for $z\in\mathbb{R}$, 
	\begin{align*}
		f_Z(z)
        &=\frac{\gamma_1^{2m_1+1}\alpha_2\cdots\alpha_N\eta}{2^{N}\pi^{N/2}\alpha_1^{2m_1}}\sum_{k=0}^{\infty}\frac{1}{(2k)!}\bigg(\frac{2\beta_1}{\alpha_1}\bigg)^{2k}G^{2N,0}_{0,2N} \bigg( \frac{\xi^2}{4^N}z^2\, \bigg|\, { - \atop k,0,\ldots,0,m_1+k,m_2, \ldots, m_N }   \bigg),
	\end{align*}
	where we used that $
    a_{j,0,\ldots,0}(z)=2^{N-2}(1+(-1)^{j})$. More generally, if $N-k$ of the skewness parameters are equal to zero then the PDF (\ref{eq:1}) reduces to a $k$-fold infinite series.
\end{remark}

\begin{proof}
We begin by considering the case $z>0$. We will need the following notation. For a random variable $V$, we denote its positive and negative parts $V^+$ and $V^-$ by $V^+=\max\{V,0\}$ and $V^-=\max\{-V,0\}$. Recall that $Z=\prod_{i=1}^NX_i$, where $X_1,\ldots,X_N$ are independent with $X_i \sim \mathrm{VG}(m_i,\alpha_i,\beta_i,0)$, $i=1,\ldots,N$. The Mellin transform of $Z^+$ is then given by 
\begin{align}
	\mathcal{M}_{f_{Z^+}}(s) &= \sum_{\sigma \in S_N^+} \bigg\{\prod_{k \in \sigma} \prod_{l \in \sigma^c }\mathcal{M}_{f_{X_k^-}}(s) \mathcal{M}_{f_{X_l^+}}(s)\bigg\}, \label{eq:2}
\end{align}
where the Mellin transforms for $X_l^+$ and $X_k^-$, which were obtained by \cite{gl24}, are given by
\begin{align}
\mathcal{M}_{f_{X_l^+}}(s)&=\frac{(\gamma_l/\alpha_l)^{2m_l+1}}{2\sqrt{\pi}\Gamma(m_l+1/2)}\sum_{j_l=0}^{\infty}\frac{(2\beta_l/\alpha_l)^{j_l}}{(j_l)!}\frac{2^s}{\alpha_l^s}\Gamma\Big(\frac{s}{2}+\frac{j_l+1}{2}\Big)\Gamma\Big(\frac{s}{2}+m_l+\frac{j_l+1}{2}\Big),\label{mel+} \\
\mathcal{M}_{f_{X_k^-}}(s)&=\frac{(\gamma_k/\alpha_k)^{2m_k+1}}{2\sqrt{\pi}\Gamma(m_k+1/2)}\sum_{j_k=0}^{\infty}\frac{(-2\beta_k/\alpha_k)^{j_k}}{(j_k)!}\frac{2^s}{\alpha_k^s}\Gamma\Big(\frac{s}{2}+\frac{j_k+1}{2}\Big)\Gamma\Big(\frac{s}{2}+m_k+\frac{j_k+1}{2}\Big).\label{mel-}
\end{align}
The PDF of $Z^+$ can be found on taking inverse Mellin transform of (\ref{eq:2}), and so we calculate
%
\begin{align}
	& \mathcal{M}^{-1} \bigg(  \frac{(2^N)^s}{\xi^s} \prod^N_{i=1} \Gamma\Big(\frac{s}{2}+\frac{j_i+1}{2}\Big) \Gamma\Big(\frac{s}{2}+m_i+\frac{j_i+1}{2}\Big) \bigg) \nonumber
	\\  &\quad=  \frac{1}{2 \pi \mathrm{i}} \lim_{U \rightarrow \infty} \int^{c+\mathrm{i}U}_{c-\mathrm{i}U} z^{-s-1}  \frac{(2^N)^{s}}{\xi^s } \prod^N_{i=1} \Gamma\Big(\frac{s}{2}+\frac{j_i+1}{2}\Big) \Gamma\Big(\frac{s}{2}+m_i+\frac{j_i+1}{2}\Big)\, \mathrm{d}s \nonumber
	\\ &\quad=  \frac{\xi}{2^{N-1}} G^{2N,0}_{0,2N} \bigg( \frac{\xi^2}{4^N}z^2\, \bigg|\, \begin{aligned} -\atop \tfrac{j_1}{2},\ldots, \tfrac{j_N}{2}, m_1+\tfrac{j_1}{2}, \ldots,  m_N+\tfrac{j_N}{2}  \end{aligned}   \bigg),\label{eq:3}
\end{align}
where $c > 0$. Here we calculated the contour integral by using  the contour integral representation (\ref{mdef}) of the $G$-function. The desired formula for the PDF of $Z$ for $z>0$ now follows from (\ref{eq:2}), (\ref{mel+}), (\ref{mel-}) and (\ref{eq:3}). For $z<0$, we use that the Mellin transform of $Z^-$ is given by 
\begin{align}
	\mathcal{M}_{f_{Z^-}}(s)=\sum_{\sigma\in S_N^-}\bigg\{ \prod_{k \in \sigma} \prod_{l \in \sigma^c} \mathcal{M}_{f_{X_k^-}}(s)\mathcal{M}_{f_{X_l^+}}(s)\bigg\},\nonumber
\end{align}
and by a similar argument we obtain the desired formula (\ref{eq:1}) for the PDF of $Z$ for $z<0$. 
\end{proof}

We now provide a simplified formula for the PDF of $Z=\prod_{i=1}^{N}X_i$ in the special case that $m_1,\ldots,m_N$ are half-integers. In this case, the PDF can be expressed as an $N$-fold finite sum rather than an $N$-fold infinite series. The Meijer $G$-functions in the series are also of lower order. 
\begin{proposition}\label{prop2.3}
	Suppose that $m_1-1/2,\ldots,m_N-1/2\geq0$ are non-negative integers. Then
	\begin{align}
		f_Z(z)&=\bigg\{\prod_{i=1}^{N}\frac{\gamma_i^{2m_i+1}}{(2\alpha_i)^{m_i+1/2}(m_i-1/2)!}\bigg\}\sum_{j_1=0}^{m_1-1/2}\cdots\sum_{j_N=0}^{m_N-1/2}\frac{(m_1-1/2+j_1)!}{(m_1-1/2-j_1)!j_1!}\times\cdots\nonumber\\
		&\quad\times\frac{(m_N-1/2+j_N)!}{(m_N-1/2-j_N)!j_N!}\sum_{\sigma\in A(z)}\bigg\{\prod_{k \in \sigma}\prod_{l \in \sigma^c}\frac{(\alpha_k+\beta_k)^{j_k+1/2-m_k}}{(2\alpha_k)^{j_k}}\cdot\frac{(\alpha_l-\beta_l)^{j_l+1/2-m_l}}{(2\alpha_l)^{j_l}}\bigg\}\nonumber\\
		&\quad\times G^{N,0}_{0,N} \bigg( \omega_\sigma|z|\, \bigg|\, \begin{aligned} -\atop  m_1-\tfrac{1}{2}-j_1, \ldots, m_N-\tfrac{1}{2}-j_N \end{aligned}   \bigg),\quad z\in\mathbb{R}, \label{redf}
	\end{align}
	where $A(z)=S_N^+$ for $z>0$ and $A(z)=S_N^-$ for $z<0$.
	\end{proposition}
	\begin{remark}
			When $\beta_1=\cdots=\beta_N=0$ the PDF (\ref{redf}) simplifies to the following. For $z\in\mathbb{R}$,
		\begin{align}
			f_Z(z)&=\frac{\xi}{2}\bigg\{\prod_{i=1}^{N}\frac{2^{1/2-m_i}}{(m_i-1/2)!}\bigg\}\sum_{j_1=0}^{m_1-1/2}\cdots\sum_{j_N=0}^{m_N-1/2}\frac{(m_1-1/2+j_1)!}{(m_1-1/2-j_1)!j_1!2^{j_1}}\times\cdots\nonumber\\
			&\quad\times\frac{(m_N-1/2+j_N)!}{(m_N-1/2-j_N)!j_N!2^{j_N}}G^{N,0}_{0,N} \bigg( \xi|z|\,\bigg|\, \begin{aligned} -\atop m_1-\tfrac{1}{2}-j_1, \ldots, m_N-\tfrac{1}{2}-j_N \end{aligned}   \bigg)\label{beta0}.
		\end{align}
	\end{remark}
        \begin{proof}
		We first consider the case $z>0$. We begin by computing the Mellin transform of $X_i^+$:
		\begin{align}
			\mathcal{M}f_{X_i^+}(s)&=\int_{0}^{\infty}x^s\frac{\gamma_i^{2m_i+1}}{\sqrt{\pi}(2\alpha_i)^{m_i}(m_i-1/2)!}\mathrm{e}^{\beta_ix}x^{m_i}K_{m_i}(\alpha_ix)\,\mathrm{d}x\nonumber\\
			&=\frac{\gamma_i^{2m_i+1}}{(2\alpha_i)^{m_i+1/2}(m_i-1/2)!}\sum_{j_i=0}^{m_i-1/2}\frac{(m_i-1/2+j_i)!}{(m_i-1/2-j_i)!j_i!}(2\alpha_i)^{-j_i}\nonumber\\
			&\quad\times\int_{0}^{\infty}\mathrm{e}^{-(\alpha_i-\beta_i)x}x^{m_i-1/2-j_i+s}\,\mathrm{d}x\nonumber\\
			&=\frac{\gamma_i^{2m_i+1}}{(2\alpha_i)^{m_i+1/2}(m_i-1/2)!}\sum_{j_i=0}^{m_i-1/2}\frac{(m_i-1/2+j_i)!}{(m_i-1/2-j_i)!j_i!}(2\alpha_i)^{-j_i}\nonumber\\
			&\quad\times(\alpha_i-\beta_i)^{j_i-m_i-1/2}\frac{\Gamma(m_i+s+1/2-j_i)}{(\alpha_i-\beta_i)^s},\nonumber
		\end{align}
		where in the second step we used the elementary representation of the modified Bessel function of the second kind ${K}	_{m+1/2}(x)=\sqrt{\pi/(2x)}\sum_{j=0}^m\frac{(m+j)!}{(m-j)!j!}(2x)^{-j}\mathrm{e}^{-x}$ for $m=0,1,2,\ldots$,  and in the final step we used the well-known integral representation of the gamma function to evaluate the integral. 
		Similarly, for $X_i^-$ we have 
		\begin{align*}
			\mathcal{M}f_{X_i^-}(s)&=\frac{\gamma_i^{2m_i+1}}{(2\alpha_i)^{m_i+1/2}(m_i-1/2)!}\sum_{j_i=0}^{m_i-1/2}\frac{(m_i-1/2+j_i)!}{(m_i-1/2-j_i)!j_i!}(2\alpha_i)^{-j_i}\nonumber\\
			&\quad\times(\alpha_i+\beta_i)^{j_i-m_i-1/2}\frac{\Gamma(m_i+s+1/2-j_i)}{(\alpha_i+\beta_i)^s}.
		\end{align*}
		The Mellin transform of $Z^+$ is given by
		\begin{align}
			\mathcal{M}f_{Z^+}(s)
			&=\sum_{\sigma \in S_N^+}\bigg\{\prod_{k \in \sigma}\prod_{l \in \sigma^c}\mathcal{M}_{f_{X_k^-}}(s)\mathcal{M}_{f_{X_l^+}}(s)\bigg\}\nonumber\\
			&=\bigg\{\prod_{i=1}^{N}\frac{\gamma_i^{2m_i+1}}{(2\alpha_i)^{m_i+1/2}(m_i-1/2)!}\bigg\}\sum_{j_1=0}^{m_1-1/2}\cdots\sum_{j_N=0}^{m_N-1/2}\frac{(m_1-1/2+j_1)!}{(m_1-1/2-j_1)!j_1!}\times\cdots\nonumber\\
			&\quad\times\frac{(m_N-1/2+j_N)!}{(m_N-1/2-j_N)!j_N!}\sum_{\sigma\in S_N^+}\bigg\{\prod_{k \in \sigma}\prod_{l \in \sigma^c}\frac{(\alpha_k+\beta_k)^{j_k-1/2-m_k}}{(2\alpha_k)^{j_k}}\cdot\frac{(\alpha_l-\beta_l)^{j_l-1/2-m_l}}{(2\alpha_l)^{j_l}}\nonumber\\
			&\quad\times\frac{\Gamma(m_k+s+1/2-j_k)}{(\alpha_k-\beta_k)^{s}}\frac{\Gamma(m_l+s+1/2-j_l)}{(\alpha_l+\beta_l)^{s}}\bigg\}.\label{zplus}
		\end{align}
		Taking the inverse Mellin transform of (\ref{zplus}) yields the PDF of $Z^+$ and we are left to find
		\begin{align*}
			&\mathcal{M}^{-1}\bigg(\prod_{k \in \sigma}\prod_{l \in \sigma^c }\frac{\Gamma(m_k+s+1/2-j_k)}{(\alpha_k+\beta_k)^s}\frac{\Gamma(m_l+s+1/2-j_l)}{(\alpha_l-\beta_l)^s}\bigg)\nonumber\\
			&=\mathcal{M}^{-1}\bigg(\bigg\{\prod_{i=1}^{N}\Gamma(m_i+s+1/2-j_i)\bigg\}\bigg\{\prod_{k \in \sigma}\prod_{l \in \sigma^c}(\alpha_k+\beta_k)^{-s}(\alpha_l-\beta_l)^{-s}\bigg\}\bigg)\nonumber\\
			&=\frac{1}{2\pi\mathrm{i}}\lim_{U \rightarrow \infty}\int_{c+\mathrm{i}U}^{c-\mathrm{i}U}z^{-s-1}\bigg\{\prod_{i=1}^{N}\Gamma(m_i+s+1/2-j_i)\bigg\}\bigg\{\prod_{k\in\sigma}\prod_{l \in \sigma^c}(\alpha_k+\beta_k)^{-s}(\alpha_l-\beta_l)^{-s}\bigg\}\,\mathrm{d}s\nonumber\\
			&=\bigg\{\prod_{k\in\sigma}\prod_{l \in \sigma^c}(\alpha_k+\beta_k)(\alpha_l-\beta_l)\bigg\}\nonumber\\
			&\quad\times G^{N,0}_{0,N} \bigg( z\prod_{k \in \sigma}\prod_{l \in \sigma^c}(\alpha_k+\beta_k)(\alpha_l-\beta_l)\,\bigg| \,\begin{aligned} -\atop  m_1-\tfrac{1}{2}-j_1, \ldots, m_N-\tfrac{1}{2}-j_N \end{aligned}\bigg),
		\end{align*}
		where $c>0$ and in the final step we evaluated the contour integral by using a change of variables and the contour integral representation (\ref{mdef}) of the Meijer $G$-function. We now have
		\begin{align}
			f_{Z^+}(z)&=\bigg\{\prod_{i=1}^{N}\frac{\gamma_i^{2m_i+1}}{(2\alpha_i)^{m_i+1/2}(m_i-1/2)!}\bigg\}\sum_{j_1=0}^{m_1-1/2}\cdots\sum_{j_N=0}^{m_N-1/2}\frac{(m_1-1/2+j_1)!}{(m_1-1/2-j_1)!j_1!}\cdots\nonumber\\
			&\quad\times\frac{(m_N-1/2+j_N)!}{(m_N-1/2-j_N)!j_N!}\sum_{\sigma\in S_N^+}\bigg\{\prod_{k \in \sigma}\prod_{l \in \sigma^c}\frac{(\alpha_k+\beta_k)^{j_k+1/2-m_k}}{(2\alpha_k)^{j_k}}\cdot\frac{(\alpha_l-\beta_l)^{j_l+1/2-m_l}}{(2\alpha_l)^{j_l}}\bigg\}\nonumber\\
			&\quad\times G^{N,0}_{0,N} \bigg(\omega_\sigma z \,\bigg|\, \begin{aligned} -\atop  m_1-\tfrac{1}{2}-j_1, \ldots, m_N-\tfrac{1}{2}-j_N \end{aligned}\bigg).\nonumber
		\end{align}
		By a similar argument we can obtain the desired formula for the PDF of $Z$ for $z<0$.
        \end{proof}
\subsection{Cumulative distribution function}\label{sec2.2}
Denote the CDF of $Z = \prod^N_{i=1} X_i $ by $F_Z(z) = \mathbb{P}(Z \leq z) $. We now obtain an exact formula for the CDF of the product of $N$ independent symmetric VG random variables.
\begin{theorem}
	Suppose that $\beta_1=\cdots=\beta_N=0$ and $m_1,\ldots,m_N>-1/2$, $\alpha_1,\ldots,\alpha_N>0$. Then, for $z \in \mathbb{R}$,
	\begin{align}F_Z(z)&=\frac{1}{2}+ \frac{\xi\eta z }{  2^{N+1} \pi^{N/2}  } G^{2N,1}_{1,2N+1}\bigg( \frac{\xi^2}{4^{N}}z^2 \,\bigg|\,{ \frac{1}{2} \atop 0,\ldots,0,m_1,\ldots,m_N,-\frac{1}{2}}\bigg)\label{cdf1}\\
		&=\frac{1}{2}+ \frac{\eta\,\mathrm{sgn}(z)}{  2 \pi^{N/2}  } G^{2N,1}_{1,2N+1}\bigg( \frac{\xi^2}{4^{N}}z^2\, \bigg|\,{ 1 \atop \frac{1}{2},\ldots,\frac{1}{2},m_1+\frac{1}{2},\ldots,m_N+\frac{1}{2},{0}}\bigg)\label{cdf2}.
	\end{align}
\end{theorem}

\begin{proof} 
	We only consider the case $z>0$, since the PDF (\ref{simp}) is symmetric about the origin when $\beta_1=\cdots=\beta_N=0$. Now, for $z>0$, 
	\begin{align}
		F_Z(z)&=\frac{1}{2}+ \frac{\xi \eta }{ 2^N \pi^{N/2}  } \int_0^z G^{2N,0}_{0,2N}\bigg( \frac{\xi^2 }{2^{2N}} y^2\, \bigg|\,{- \atop 0,\ldots,0,m_1,\ldots,m_N}\bigg)\,\mathrm{d}y  \nonumber
		\\&=\frac{1}{2}+ \frac{\eta }{ 2\pi^{N/2}  } \int_0^{\tfrac{\xi^2z^2}{2^{2N}}}u^{-1/2} G^{2N,0}_{0,2N}\bigg( u \, \bigg|\,{- \atop 0,\ldots,0,m_1,\ldots,m_N}\bigg)\,\mathrm{d}u \nonumber\\
		&=\frac{1}{2}+\frac{\eta}{2\pi^{N/2}}\bigg[u^{1/2}G^{2N,1}_{1,2N+1}\bigg( u \, \bigg|\,{\tfrac{1}{2} \atop 0,\ldots,0,m_1,\ldots,m_N, -\tfrac{1}{2}}\bigg)\bigg]^{\tfrac{\xi^2z^2}{2^{2N}}}_{0} \label{sym}
		\\ & = \frac{1}{2}+ \frac{\xi \eta z }{  2^{N+1} \pi^{N/2}  } G^{2N,1}_{1,2N+1}\bigg( \frac{\xi^2}{4^{N}}z^2\, \bigg|\,{ \frac{1}{2} \atop 0,\ldots,0,m_1,\ldots,m_N,-\frac{1}{2}}\bigg)\nonumber .
	\end{align}
	Here we evaluated the integral using (\ref{mint}) and in the last step 
	used that the $G$-function in (\ref{sym}) evaluated at $u = 0$ is equal to zero, which follows easily from the contour integral definition (\ref{mdef}) of the Meijer $G$-function. Lastly, (\ref{cdf2}) follows from (\ref{cdf1}) due to the relation (\ref{meijergidentity}). \end{proof}

When $m_1,\ldots,m_N$ are half-integers, the CDF of $Z$ can be expressed as an $N$-fold finite sum involving Meijer $G$-functions of lower order.

\begin{proposition}\label{p2.7}
	Suppose that $m_1-1/2,\ldots,m_N-1/2\geq0$ are non-negative integers. 
    Also, let $\beta_1=\cdots=\beta_N=0$ and $\alpha_1,\ldots,\alpha_N>0$. Then, for $z \in \mathbb{R}$,
	\begin{align}
		F_Z(z)&=\frac{1}{2}+\frac{\mathrm{sgn}(z)}{2}\bigg\{\prod_{i=1}^{N}\frac{2^{1/2-m_i}}{(m_i-1/2)!}\bigg\}\sum_{j_1=0}^{m_1-1/2}\cdots\sum_{j_N=0}^{m_N-1/2}\frac{(m_1-1/2+j_1)!}{(m_1-1/2-j_1)!j_1!2^{j_1}}\times\cdots\nonumber\\
		&\quad\times\frac{(m_N-1/2+j_N)!}{(m_N-1/2-j_N)!j_N!2^{j_N}} G^{N,1}_{1,N+1} \bigg( \xi|z|\,\bigg|\, \begin{aligned} &1 \\ m_1+\tfrac{1}{2}-j_1, \ldots&, m_N+\tfrac{1}{2}-j_N,0 \end{aligned}   \bigg).\label{p2.73}
	\end{align}
\end{proposition}
\begin{proof} 
Since $\beta_1=\cdots=\beta_N=0$, the PDF of $Z$, as given by (\ref{beta0}), is symmetric about the origin, and therefore $F_Z(z)=1/2+\mathrm{sgn}(z)\int_{0}^{|z|}f_Z(y)\,\mathrm{d}y$.
We then have, for $z\in\mathbb{R}$,
\begin{align}
\!\!\!F_Z(z)&=\frac{1}{2}+\frac{\mathrm{sgn}(z)}{2}\bigg\{\prod_{i=1}^N\frac{2^{1/2-m_i}}{(m_i-1/2)!}\bigg\}\sum_{j_1=0}^{m_1-1/2}\cdots\sum_{j_N=0}^{m_N-1/2}\frac{(m_1-1/2+j_1)!}{(m_1-1/2-j_1)!j_1!2^{j_1}}\times\cdots\nonumber\\
&\quad\times\frac{(m_N-1/2+j_N)!}{(m_N-1/2-j_N)!j_N!2^{j_N}}\int_{0}^{\xi|z|}G_{0,N}^{N,0}\bigg(u \; \bigg| \;{-\atop m_1-\frac{1}{2}-j_1,\ldots,m_N-\frac{1}{2}-j_N}\bigg)\,\mathrm{d}u\nonumber\\
&=\frac{1}{2}+\frac{\mathrm{sgn}(z)}{2}\bigg\{\prod_{i=1}^{N}\frac{2^{1/2-m_i}}{(m_i-1/2)!}\bigg\}\sum_{j_1=0}^{m_1-1/2}\cdots\sum_{j_N=0}^{m_N-1/2}\frac{(m_1-1/2+j_1)!}{(m_1-1/2-j_1)!j_1!2^{j_1}}\times\cdots\nonumber\\
&\quad\times\frac{(m_N-1/2+j_N)!}{(m_N-1/2-j_N)!j_N!2^{j_N}}\bigg[G_{1,N+1}^{N,1}\bigg(u \; \bigg| \;{1\atop m_1+\frac{1}{2}-j_1,\ldots,m_N+\frac{1}{2}-j_N,0}\bigg)\bigg]^{\xi|z|}_0\label{mg}\\
&=\frac{1}{2}+\frac{\mathrm{sgn}(z)}{2}\bigg\{\prod_{i=1}^{N}\frac{2^{1/2-m_i}}{(m_i-1/2)!}\bigg\}\sum_{j_1=0}^{m_1-1/2}\cdots\sum_{j_N=0}^{m_N-1/2}\frac{(m_1-1/2+j_1)!}{(m_1-1/2-j_1)!j_1!2^{j_1}}\times\cdots\nonumber\\
&\quad\times\frac{(m_N-1/2+j_N)!}{(m_N-1/2-j_N)!j_N!2^{j_N}}G_{1,N+1}^{N,1}\bigg(\xi |z| \; \bigg| \;{1\atop m_1+\frac{1}{2}-j_1,\ldots,m_N+\frac{1}{2}-j_N,0}\bigg)\nonumber,
\end{align}
where in the third step we evaluated the integral using (\ref{mint1}). The final step follows from the fact that the  $G$-function in (\ref{mg}) evaluated at $u=0$ is equal to zero, which can be readily seen from the contour integral representation (\ref{mdef}) of the Meijer $G$-function. 
\end{proof}

In the following proposition, we provide a closed-form formula for the probability $\mathbb{P}(Z\leq0)$, expressed in terms of the Gaussian hypergeometric function (see \cite[Chapter 15]{olver} for a definition). We make use of the following formula of \cite{gaunt on}. Let $m>-1/2$, $0\leq|\beta|<\alpha$. Then, for $X\sim \mathrm{VG}(m,\alpha,\beta,0)$,
\begin{equation*}
	\mathbb{P}(X\leq0)=P_{m,\alpha,\beta}:=\frac{1}{2}-\frac{(\beta/\alpha)\Gamma(m+1)}{\sqrt{\pi}\Gamma(m+1/2)}\bigg(1-\frac{\beta^2}{\alpha^2}\bigg)^{m+1/2} {_2F_1}\bigg(1,m+1;\frac{3}{2};\frac{\beta^2}{\alpha^2}\bigg).\label{2.16}
\end{equation*}

\begin{proposition}
	Let $m_1,\ldots,m_N>-1/2$, $0\leq|\beta_i|<\alpha_i$, $i=1,\ldots,N$. Then
	\begin{align}
		\mathbb{P}(Z\leq0)=\sum_{i=1}^{N}P_{m_i,\alpha_i,\beta_i}+\sum_{\substack{\sigma\in S_N: \\|\sigma|\geq2 } }(-2)^{|\sigma|-1}\prod_{j \in \sigma}P_{m_j,\alpha_j,\beta_j}.\label{prob}
	\end{align}
In the case $(m_i,\alpha_i,\beta_i)=(m,\alpha,\beta)$ for $i=1,\ldots,N$, we have the simpler formula
\begin{align}\label{prob2}
   \mathbb{P}(Z\leq0)=\frac{1}{2}-\frac{1}{2}(1-2P_{m,\alpha,\beta})^N. 
\end{align}
\end{proposition}
\begin{proof} A simple calculation shows that, for independent continuous random variables $U$ and $V$, 
\begin{align}\label{uveqn}
\mathbb{P}(UV\leq0)=\mathbb{P}(U\leq0)+\mathbb{P}(V\leq0)-2\mathbb{P}(U\leq0)\mathbb{P}(V\leq0).
\end{align} 
Therefore in the $N=2$ case we have that $\mathbb{P}(Z\leq0)=P_{m_1,\alpha_1,\beta_1}+P_{m_2,\alpha_2,\beta_2}-2P_{m_1,\alpha_1,\beta_1}P_{m_2,\alpha_2,\beta_2}$. This confirms that equation (\ref{prob}) holds for $N=2$. The general case follows by a straightforward induction on $N$ using equation (\ref{uveqn}).
The proof of formula (\ref{prob2}) in the case $(m_i,\alpha_i,\beta_i)=(m,\alpha,\beta)$, $i=1,\ldots,N$, follows from a similar inductive argument.
\end{proof}

\subsection{Characteristic function}

In the following theorem, we obtain a closed-form formula for the characteristic function of the product of independent symmetric VG random variables. We write $\varphi_Z(t)=\mathbb{E}[\mathrm{e}^{\mathrm{i}tZ}]$ for the characteristic function of  $Z$. 

\begin{theorem}Let $\beta_1=\cdots=\beta_N=0$ and $m_1,\ldots,m_N>-1/2$, $\alpha_1,\ldots,\alpha_N>0$. Then
\begin{align}
	\varphi_Z(t)&=\frac{\xi\eta}{2^{N-1}\pi^{(N-1)/2}}|t|^{-1}G^{2N-1,1}_{1,2N-1} \bigg( \frac{\xi^2}{4^{N-1}t^2} \,\bigg|\, \begin{aligned} \tfrac{1}{2}\atop  0, \ldots,0, \,m_1,\ldots, m_N  \end{aligned}   \bigg),\quad t\in\mathbb{R}.
    \label{char1}
\end{align}
\end{theorem}

\begin{proof}As the PDF (\ref{simp}) of $Z$ is an even function in $z$, $\varphi_Z(t)=\mathbb{E}[\cos(tZ)]$. We then get that
\begin{align}
	\varphi_Z(t)&=2\int_{0}^{\infty}\frac{\xi\eta}{2^N\pi^{N/2}}\cos(tz)G^{2N,0}_{0,2N} \bigg( \frac{\xi^2}{4^N}z^2\, \bigg|\, \begin{aligned} -\atop 0, \ldots,0, \,m_1,\ldots, m_N  \end{aligned}   \bigg)\,\mathrm{d}z \nonumber\\
	&=\frac{\xi\eta}{2^{N-1}\pi^{N/2}}\frac{\sqrt{\pi}}{|t|}G^{2N,1}_{2,2N} \bigg( \frac{\xi^2}{4^{N-1}t^2}\, \bigg|\, \begin{aligned} {\tfrac{1}{2},0}\atop 0, \ldots,0, \,m_1,\ldots, m_N  \end{aligned}   \bigg)\nonumber\\
    &=\frac{\xi\eta}{2^{N-1}\pi^{(N-1)/2}}|t|^{-1}G^{2N-1,1}_{1,2N-1} \bigg( \frac{\xi^2}{4^{N-1}t^2} \,\bigg|\, { \tfrac{1}{2}\atop  0, \ldots,0, \,m_1,\ldots, m_N}\bigg),\nonumber
\end{align}
where the integral was evaluated using equation 5.6.3(18) of \cite{luke} and the third equality was obtained by using (\ref{lukeformula}).
\end{proof}
When $m_1,\ldots,m_N$ are half-integers, the characteristic function
can be expressed as an $N$-fold finite sum involving Meijer $G$-functions of lower order, even for non-zero skewness parameters.
\begin{proposition}
Suppose that $m_1-1/2,\ldots,m_N-1/2\geq0$ are non-negative integers and that $0\leq|\beta_i|<\alpha_i$, $i=1,\ldots,N$. Then
\begin{align}
	\varphi_Z(t)&=\frac{\mathrm{i}}{t}\bigg\{\prod_{i=1}^{N}\frac{\gamma_i^{2m_i+1}}{(2\alpha_i)^{m_i+1/2}(m_i-1/2)!}\bigg\}\sum_{j_1=0}^{m_1-1/2}\cdots\sum_{j_N=0}^{m_N-1/2}\frac{(m_1-1/2+j_1)!}{(m_1-1/2-j_1)!j_1!}\times\cdots\nonumber\\
	&\quad\times\frac{(m_N-1/2+j_N)!}{(m_N-1/2-j_N)!j_N!}\bigg\{\sum_{\sigma\in S_N^+}\bigg\{\prod_{k \in \sigma}\prod_{l \in \sigma^c}\frac{(\alpha_k+\beta_k)^{j_k+1/2-m_k}}{(2\alpha_k)^{j_k}}\cdot\frac{(\alpha_l-\beta_l)^{j_l+1/2-m_l}}{(2\alpha_l)^{j_l}}\bigg\}\nonumber\\
	&\quad\times G^{N,1}_{1,N} \bigg(\frac{\mathrm{i}\omega_\sigma}{t}\, \bigg|\, \begin{aligned} &0 \\ m_1-\tfrac{1}{2}-j_1, \ldots&,m_N-\tfrac{1}{2}-j_N  \end{aligned}   \bigg)\nonumber\\
	&\quad-\sum_{\sigma\in S_N^-}\bigg\{\prod_{k \in \sigma}\prod_{l \in \sigma^c}\frac{(\alpha_k+\beta_k)^{j_k+1/2-m_k}}{(2\alpha_k)^{j_k}}\cdot\frac{(\alpha_l-\beta_l)^{j_l+1/2-m_l}}{(2\alpha_l)^{j_l}}\bigg\}\nonumber\\
	&\quad\times G^{N,1}_{1,N} \bigg(-\frac{\mathrm{i}\omega_\sigma}{t}\, \bigg|\, \begin{aligned} &0 \\ m_1-\tfrac{1}{2}-j_1, \ldots&,m_N-\tfrac{1}{2}-j_N  \end{aligned}   \bigg)\bigg\},\quad t\in\mathbb{R}.\label{Cfhalf}
\end{align}
\end{proposition}
\begin{proof}
Recall that under the conditions of the proposition a formula for the PDF of $Z$ is given by (\ref{redf}). We then calculate $\varphi_Z(t)=\mathbb{E}[\mathrm{e}^{\mathrm{i}tZ}]=\int_{0}^{\infty}\mathrm{e}^{\mathrm{i}tz}f_Z(z)\,\mathrm{d}z+\int_{-\infty}^{0}\mathrm{e}^{\mathrm{i}tz}f_Z(z)\,\mathrm{d}z$, evaluating the integrals using formula 5.6.3(1) of \cite{luke}. 
\end{proof}

\subsection{Asymptotic behaviour of the distribution}

Given that the PDF of the product $Z$ for general parameter values takes a rather complicated form, it is of interest to study the asymptotic behaviour of the distribution. In this section, we derive asymptotic approximations for the PDF, tail probabilities and quantile function. We recall that we use the convention that the empty product is set to 1.

\subsubsection{Asymptotic behaviour of the density at the origin}

\begin{proposition}\label{vg0}
Let $m_1,\ldots,m_N>-1/2$, $0\leq|\beta_i|<\alpha_i$, $i=1,\ldots,N$, where $N\geq2$. 

\vspace{2mm}

\noindent{(i)} Suppose that $m_1,\ldots,m_N\geq0$, and let $ 0 \leq t \leq N$ be the number of zeros among them. Without loss of generality, set $m_1= \cdots =m_t=0$ and $m_{t+1},\ldots, m_N > 0.$ Then, as $z \rightarrow 0,$ 
\begin{align}
	f_Z(z) \sim  \frac{2^{t-1} \gamma_1\cdots\gamma_t\gamma_{t+1}^{2m_{t+1}+1}\cdots\gamma_N^{2m_N+1}  \eta  }{  (N+t-1)! \pi^{N/2} \alpha_{t+1}^{2m_{t+1}}\cdots\alpha_N^{2m_N} } \bigg\{ \prod^{N-t}_{i=1} \Gamma(m_{t+i}) \bigg\} (- \ln|z|)^{N+t-1}.\label{lim1}
\end{align}

\noindent{(ii)} Suppose that $m_1 = \min\{m_1,\ldots,m_N\} < 0.$ Without loss of generality, set $m_1= \cdots =m_t<m_{t+1},\ldots,m_N$ for $1\leq t\leq N$. Then, as $z \rightarrow 0$, 
\begin{align}
	f_Z(z) &\sim \frac{\gamma_1^{2m_1+1}\cdots\gamma_N^{2m_N+1} (\Gamma(-m_1))^N\eta }{2^{ N (1+2m_1)} \pi^{N/2}  } \bigg( \frac{\alpha_{t+1}^{2m_1}\cdots\alpha_N^{2m_1} }{\alpha_{t+1}^{2m_{t+1}}\cdots\alpha_N^{2m_N} } \bigg) \bigg\{\prod^N_{i=t+1} \Gamma(m_i-m_1)\bigg\}\nonumber\\
	&\quad\times\frac{(-2\ln|z|)^{t-1}}{(t-1)!}|z|^{2m_1}.\label{lim2}
\end{align}

\noindent (iii) When $N\geq2$, the distribution of $Z$ is unimodal with mode 0 for all parameter constellations.
\end{proposition}

\begin{proof}

Recall that the density function (\ref{eq:1}) for the product $Z$ is an infinite summation of $G$-functions. Set $x=\xi z/2^N$. We will derive the limiting forms (\ref{lim1}) and (\ref{lim2}) by considering the
asymptotic behaviour of the functions
\begin{equation*}
	g_{j_1,\ldots, j_N, m_1,\ldots,m_N}(x)=G^{2N,0}_{0,2N} \bigg(x^2\, \bigg|\, \begin{aligned} -\atop  \tfrac{j_1}{2},\ldots,\tfrac{j_N}{2}, m_1+\tfrac{j_1}{2}, \ldots, m_N+\tfrac{j_N}{2}  \end{aligned}   \bigg),
\end{equation*}
in the limit $x\downarrow0$, where $j_1,\ldots,j_N\geq0$. Note that since $g_{j_1,\ldots, j_N, m_1,\ldots,m_N}(x)$ is an even function in $x$ it does indeed suffice to consider the limit $x\downarrow0$. We can express $g_{j_1,\ldots, j_N, m_1,\ldots,m_N}(x)$ as a contour integral using (\ref{mdef}):
\begin{align}
	g_{j_1,\ldots, j_N, m_1,\ldots,m_N}(x)&=\frac{1}{2\pi\mathrm{i}}\int_{L}\bigg\{\prod_{i=1}^{N}\Gamma\Big(s+\frac{j_i}{2}\Big)\Gamma\Big(s+m_i+\frac{j_i}{2}\Big)\bigg\}x^{-2s}\,\mathrm{d}s \nonumber \\
    &=\sum_{s^*\in B}\mathrm{Res}\bigg[\bigg\{\prod_{i=1}^{N}\Gamma\Big(s+\frac{j_i}{2}\Big)\Gamma\Big(s+m_i+\frac{j_i}{2}\Big)\bigg\}x^{-2s}, \,s=s^*\bigg],\label{res}
\end{align}
where the contour $L$ is a loop encircling the poles of the gamma functions and $B$ is the set of these poles.  
We compute the residues in (\ref{res}) for the two cases (\text{i}) and (ii) separately.

\vspace{2mm}

\noindent{(i)} Suppose $m_1,\ldots,m_N\geq0$ with $m_1=\cdots=m_t=0$ and $m_{t+1},\ldots ,m_N>0$. We will first consider the case $j_1=\cdots=j_N=0$. We calculate
\begin{align}
	\mathrm{Res}\bigg[(\Gamma(s))^{t+N}\bigg\{\prod_{i=t+1}^{N}\Gamma(s+m_i)\bigg\}x^{-2s},\, s=s^*\bigg]\label{res1}
\end{align}
for all poles $s^*$ of the gamma functions in (\ref{res1}). We begin with the pole at $s=0$. 
 On using the expansion
$x^{-2s} =\exp(-2s\ln(x)) = \sum^\infty_{i=0} (-2 \ln (x))^i  s^i/i!$, and the asymptotic expansion $\Gamma(s)=1/s-\gamma +o(1)$ as $s\rightarrow0$ ($\gamma$ is the Euler-Mascheroni constant), we obtain that,
as $x\downarrow0$, 
\begin{align}
	\mathrm{Res}\bigg[(\Gamma(s))^{t+N}\bigg\{\prod_{i=t+1}^{N}\Gamma(s+m_i)\bigg\}x^{-2s},\, s=0\bigg]\sim\bigg\{\prod_{i=t+1}^{N}\Gamma(m_{i})\bigg\}\frac{(-2\ln(x))^{N+t-1}}{(N+t-1)!}.\label{res2}
\end{align}
 It is readily seen that for all other poles the residue (\ref{res1}) is of a smaller asymptotic
order in the limit $x\downarrow0$ than (\ref{res2}). It is also easily seen that for all other choices of $j_1,\ldots,j_N\geq0$, besides the case we considered of $j_1=\cdots=j_N=0$, that the residues in (\ref{res}) are of a smaller asymptotic order in the limit $x\downarrow0$ than  (\ref{res1}). Thus, as $x\rightarrow0$ (equivalently as $z\rightarrow0$),
\begin{align}
	g_{0,\ldots, 0, m_{t+1}, \ldots, m_N}(x)
    &\sim\bigg\{\prod_{i=t+1}^{N}\Gamma(m_{i})\bigg\}\frac{(-2\ln|x|)^{N+t-1}}{(N+t-1)!}\nonumber\\
    &\sim\bigg\{\prod_{i=t+1}^{N}\Gamma(m_{i})\bigg\}\frac{(-2\ln|z|)^{N+t-1}}{(N+t-1)!}.\label{glim}
\end{align}
Thus, since $g_{j_1,\ldots, j_N, m_1,\ldots,m_N}(x)$ is of smaller asymptotic order in the limit $x\rightarrow0$ (equivalently $z\rightarrow0$) for all other choices of $j_1,\ldots,j_N\geq0$, on multiplying (\ref{glim}) through by the multiplicative constant of $g_{0,\ldots, 0, m_{t+1}, \ldots, m_N}(x)$ from the PDF (\ref{eq:1}) we obtain the limiting form (\ref{lim1}). 

\vspace{2mm}

\noindent{(ii)} Suppose $m_1=\min\{m_1,\ldots,m_N\}<0$ and that $m_1=\cdots=m_t<m_{t+1},\ldots,m_N$ for $1\leq t\leq N$.
Since $m_1<0$, the dominant residue of (\ref{res}) is now  at the pole $s=-m_1$, and we have 
	\begin{align*}
		&\mathrm{Res}\bigg[(\Gamma(s))^{N}(\Gamma(s+m_1))^t\bigg\{\prod_{i=t+1}^{N}\Gamma(s+m_i)\bigg\}x^{-2s},\, s=-m_1\bigg]\\
		&\quad=x^{2m_1}\mathrm{Res}\bigg[(\Gamma(u-m_1))^{N}(\Gamma(u))^t\bigg\{\prod_{i=t+1}^{N}\Gamma(u-m_1+m_i)\bigg\}x^{-2u},\, u=0\bigg]\\
		&\quad\sim\big( \Gamma (-m_1)  \big)^N\frac{(-2\ln(x))^{t-1}}{(t-1)!} \bigg\{\prod^N_{i=t+1} \Gamma\Big(m_i-m_1\Big)\bigg\}  x^{2m_1}, \quad x\downarrow0,
	\end{align*}
and all other residues are of lower order as $x\downarrow0$.    
	On arguing similarly to in part $(i)$
    we obtain the limiting form (\ref{lim2}).

\vspace{2mm}

\noindent (iii) Parts (i) and (ii) imply that the PDF $f_Z(z)$ has a singularity at $z=0$ for all parameter values. Moreover, the PDF is bounded everywhere except for the singularity at the origin, and therefore the distribution of $Z$ is unimodal with mode 0 for all parameter constellations.
\end{proof}

\subsubsection{Tail behaviour of the distribution}

\begin{proposition}\label{p2.9} Suppose $m_1,\ldots,m_N>-1/2$, $0\leq|\beta_i|<\alpha_i$, $i=1,\ldots,N.$ Denote the complementary CDF by $\bar{F}_Z(z)=\mathbb{P}(Z>z)$. Also, let $\mu_N=N^{-1}\sum_{i=1}^Nm_i$. Then
\begin{align}
	\bar{F}_Z(z)&\sim\frac{(2\pi)^\frac{N-1}{2}\eta}{\sqrt{N}}\bigg\{\prod_{i=1}^{N}\frac{\gamma_i^{2m_i+1}}{(2\alpha_i)^{m_i+1/2}}\bigg\}z^{\mu_N-\tfrac{1}{2N}}\nonumber\\
    &\quad\times\!\sum_{\sigma \in S_N^+}\!\bigg\{\prod_{k \in \sigma}\prod_{l \in \sigma^c}(\lambda_k^+)^{\mu_N-m_k-\tfrac{N+1}{2N}}(\lambda_l^-)^{\mu_N-m_l-\tfrac{N+1}{2N}}\bigg\}\mathrm{e}^{-N\omega_\sigma^{1/N}z^{1/N}}, \:\,\,z\rightarrow\infty,\label{lim3}\\
	{F}_Z(z)&\sim\frac{(2\pi)^\frac{N-1}{2}\eta}{\sqrt{N}}\bigg\{\prod_{i=1}^{N}\frac{\gamma_i^{2m_i+1}}{(2\alpha_i)^{m_i+1/2}}\bigg\}(-z)^{\mu_N-\tfrac{1}{2N}}\nonumber\\
    &\quad\times\!\sum_{\sigma \in S_N^-}\!\bigg\{\prod_{k \in \sigma}\prod_{l \in \sigma^c}(\lambda_k^+)^{\mu_N-m_k-\tfrac{N+1}{2N}}(\lambda_l^-)^{\mu_N-m_l-\tfrac{N+1}{2N}}\bigg\}\mathrm{e}^{-N\omega_\sigma^{1/N}(-z)^{1/N}},\:\,\, z\rightarrow-\infty.\label{lim4}
\end{align}
\end{proposition}
\begin{corollary}\label{cor2.11}Suppose $m_1,\ldots,m_N>-1/2$, $0\leq|\beta_i|<\alpha_i$, $i=1,\ldots,N.$ Then
\begin{align}
	f_Z(z)&\sim\frac{(2\pi)^\frac{N-1}{2}\eta}{\sqrt{N}}\bigg\{\prod_{i=1}^{N}\frac{\gamma_i^{2m_i+1}}{(2\alpha_i)^{m_i+1/2}}\bigg\}z^{\mu_N+\tfrac{1}{2N}-1}\nonumber\\
    &\quad\times\!\sum_{\sigma \in S_N^+}\!\bigg\{\prod_{k \in \sigma}\prod_{l \in \sigma^c}(\lambda_k^+)^{\mu_N-m_k+\tfrac{1-N}{2N}}(\lambda_l^-)^{\mu_N-m_l+\tfrac{1-N}{2N}}\bigg\}\mathrm{e}^{-N\omega_\sigma^{1/N}z^{1/N}},\:\,\, z\rightarrow\infty,\label{lim5}\\
	f_Z(z)&\sim\frac{(2\pi)^\frac{N-1}{2}\eta}{\sqrt{N}}\bigg\{\prod_{i=1}^{N}\frac{\gamma_i^{2m_i+1}}{(2\alpha_i)^{m_i+1/2}}\bigg\}(-z)^{\mu_N+\tfrac{1}{2N}-1}\nonumber\\
    &\quad\times\!\sum_{\sigma \in S_N^-}\!\bigg\{\prod_{k \in \sigma}\prod_{l \in \sigma^c}(\lambda_k^+)^{\mu_N-m_k+\tfrac{1-N}{2N}}(\lambda_l^-)^{\mu_N-m_l+\tfrac{1-N}{2N}}\bigg\}\mathrm{e}^{-N\omega_\sigma^{1/N}(-z)^{1/N}},\:\,\, z\rightarrow-\infty.\label{lim6}
\end{align}
\end{corollary}
\begin{remark} The limiting forms of Proposition \ref{p2.9} and Corollary \ref{cor2.11}  take a simpler form for certain parameter constellations. To illustrate this, we let $\omega_+=\min_{\sigma\in S_N^+} \omega_\sigma$, and suppose that this minimum is attained for just a single $\sigma\in S_N^+$. We then have that $\exp(-N\omega_\sigma^{1/N}z^{1/N})\ll\exp(-N\omega_+^{1/N}z^{1/N})$, as $z\rightarrow\infty$, for all $\sigma\in S_N^+$ except for the one at which the minimum value $\omega$ is attained. In this case, the limiting form (\ref{lim3}) reduces to 
\begin{align}
	\bar{F}_Z(z)\sim\frac{(2\pi)^\frac{N-1}{2}\eta}{\sqrt{N}}\bigg\{\prod_{i=1}^{N}\frac{\gamma_i^{2m_i+1}}{(2\alpha_i)^{m_i+1/2}}\bigg\}z^{\mu_N-\tfrac{1}{2N}}\mathrm{e}^{-N\omega_+^{1/N}z^{1/N}},\quad z\rightarrow\infty\nonumber.
\end{align}



The limiting forms of the PDF and CDF of $Z$ also take a simpler form when $\beta_1=\cdots=\beta_N=0$. In this case,
\begin{align}
f_Z(z)&\sim\frac{\pi^{(N-1)/2}\xi\eta}{2^{N(\mu_N-1)+3/2}\sqrt{N}}(\xi|z|)^{\mu_N+\frac{1}{2N}-1}\mathrm{e}^{-N(\xi|z|)^{1/N}}, \quad |z|\rightarrow\infty, \label{limform1}\\    
\bar{F}_Z(z)&\sim\frac{\pi^{(N-1)/2}\eta}{2^{N(\mu_N-1)+3/2}\sqrt{N}}(\xi z)^{\mu_N-\frac{1}{2N}}\mathrm{e}^{-N(\xi z)^{1/N}}, \quad z\rightarrow\infty, \label{limform2}\\
F_Z(z)&\sim\frac{\pi^{(N-1)/2}\eta}{2^{N(\mu_N-1)+3/2}\sqrt{N}}(-\xi z)^{\mu_N-\frac{1}{2N}}\mathrm{e}^{-N(-\xi z)^{1/N}}, \quad z\rightarrow-\infty. \label{limform3}
\end{align}
These limiting forms can also be easily derived through the following alternative approach. The limiting form (\ref{limform1}) follows immediately on applying the limiting form (\ref{asymg}) to the formula (\ref{simp}) for the PDF of $Z$ in the case $\beta_1=\cdots=\beta_N=0$. To obtain (\ref{limform2}) we apply the limiting form (\ref{limform1}) to the formula $\bar{F}_Z(z)=\int_z^\infty f_Z(x) \,\mathrm{d}x$, and (\ref{limform2}) then follows since, for $a\in\mathbb{R}$ and $b,r>0$, we have that $\int_z^\infty x^a\mathrm{e}^{-bx^r}\,\mathrm{d}x\sim(br)^{-1}z^{a-r+1}\mathrm{e}^{-bz^r}$, as $z\rightarrow\infty$, which is easily proved by a standard integration by parts argument. 
The limiting form (\ref{limform3}) follows immediately from (\ref{limform2}) since the distribution of $Z$ is symmetric about the origin in the $\beta_1=\cdots=\beta_N=0$ case.

\end{remark}

We will make use of Lemma 2.1 of \cite{ad11} in proving  Proposition \ref{p2.9}.

\begin{lemma}[Arendarczyk and D\c{e}bicki \cite{ad11}]\label{lem21} Let $X_1$ and $X_2$ be independent, non-negative random variables such that $\bar{F}_{X_i}(x)\sim A_ix^{r_i}\exp(-b_ix^{a_i})$ as $x\rightarrow\infty$, for $i=1,2$.
Then
\[\bar{F}_{X_1X_2}(x)\sim  Ax^{r}\exp(-bx^{a}), \quad x\rightarrow\infty,\]
where, for $c:=a_1+a_2$,
\begin{align*}
a&=\frac{a_1a_2}{c},\quad b=b_1^{a_2/c}b_2^{a_1/c}\bigg(\Big(\frac{a_1}{a_2}\Big)^{a_2/c}+\Big(\frac{a_2}{a_1}\Big)^{a_1/c}\bigg),  \quad r=\frac{a_1a_2+2a_1r_2+2a_2r_1}{2c} \\
  A&=\frac{\sqrt{2\pi}A_1A_2}{\sqrt{c}}(a_1b_1)^{\frac{a_2-2r_1+2r_2}{2c}}(a_2b_2)^{\frac{a_1-2r_2+2r_1}{2c}}.
\end{align*}
\end{lemma}

\noindent{\emph{Proof of Proposition \ref{p2.9}.}} 
Let $Z_\sigma^+=\prod_{k \in \sigma}\prod_{l \in \sigma^c}X_k^-X_l^+$, $\sigma\in S_N^+$, and $Z_\sigma^-=\prod_{k \in \sigma}\prod_{l \in \sigma^c}X_k^-X_l^+$, $\sigma\in S_N^-$.
We can then derive the following limiting form by induction on $N$: As $z\rightarrow\infty$, 
\begin{align*}
\bar{F}_{Z_\sigma^+}(z)&\sim\frac{(2\pi)^\frac{N-1}{2}}{\sqrt{N}}\bigg\{\prod_{i=1}^{N}\frac{\gamma_i^{2m_i+1}}{(2\alpha_i)^{m_i+1/2}\Gamma(m_i+1/2)}\bigg\}z^{\mu_N-\tfrac{1}{2N}}\nonumber\\
&\quad\times\bigg\{\prod_{k \in \sigma}\prod_{l \in \sigma^c}(\lambda_k^+)^{\mu_N-m_k-\tfrac{N+1}{2N}}(\lambda_l^-)^{\mu_N-m_l-\tfrac{N+1}{2N}}\bigg\}\mathrm{e}^{-N\omega_\sigma^{1/N}z^{1/N}}, \quad\sigma\in S_N^+.
\end{align*}
The base case $N=2$ is given in the proof of Proposition 2.11 of \cite{gl24} and the general case follows by induction on $N$ using Lemma \ref{lem21}. 
Similarly, we have that, as $z\rightarrow-\infty$,
\begin{align}
{F}_{Z_\sigma^-}(z)&\sim\frac{(2\pi)^\frac{N-1}{2}}{\sqrt{N}}\bigg\{\prod_{i=1}^{N}\frac{\gamma_i^{2m_i+1}}{(2\alpha_i)^{m_i+1/2}\Gamma(m_i+1/2)}\bigg\}(-z)^{\mu_N-\tfrac{1}{2N}}\nonumber\\
&\quad\times\bigg\{\prod_{k \in \sigma}\prod_{l \in \sigma^c}(\lambda_k^+)^{\mu_N-m_k-\tfrac{N+1}{2N}}(\lambda_l^-)^{\mu_N-m_l-\tfrac{N+1}{2N}}\bigg\}\mathrm{e}^{-N\omega_\sigma^{1/N}(-z)^{1/N}}, \quad\sigma\in S_N^-.\nonumber
\end{align}
The limiting forms (\ref{lim3}) and (\ref{lim4}) are now obtained from the formulas $\bar{F}_{Z}(z)=\sum_{\sigma \in S_N^+ }\bar{F}_{Z^+_\sigma}(z)$ and ${F}_{Z}(z)=\sum_{\sigma \in S_N^-}F_{Z_\sigma^-}(z)$, respectively. 
\qed 

\vspace{3mm}

\noindent{\emph{Proof of Corollary \ref{cor2.11}.}} We obtain the limiting forms (\ref{lim5}) and (\ref{lim6}) from the limiting forms (\ref{lim3}) and (\ref{lim4}), respectively, by using the relations $f_Z(z)=-\bar{F}_Z'(z)$ and $f_Z(z)=F_Z'(z)$ and retaining the leading order terms. \qed

\subsubsection{Asymptotic behaviour of the quantile function}


\begin{proposition}\label{propq} For $0<p<1$, write $Q(p)=F_Z^{-1}(p)$ for the quantile function of $Z=\prod_{i=1}^{N}X_i$. Let $m_1,\ldots,m_N>-1/2$, $0\leq|\beta_i|<\alpha_i$, $i=1,\ldots,N$. Denote $\omega_+=\min_{\sigma\in S_N^+} \omega_\sigma$ and $\omega_-=\min_{\sigma\in S_N^-} \omega_\sigma$. Then 
	\begin{align}
		Q(p)&\sim\frac{1}{N^N\omega_+}\big(-\ln(1-p)\big)^N,\quad p\rightarrow1, \label{Q1}\\
		Q(p)&\sim-\frac{1}{N^N\omega_-}\big(-\ln(p)\big)^N,\quad p\rightarrow0.\label{Q2}
	\end{align}
\end{proposition}

We begin by proving the following lemma.
\begin{lemma} Let $a,A,N,z>0$ and $r\in\mathbb{R}$. Let $g:(0,\infty)\rightarrow\mathbb{R}$ be a function such that $g(x)\rightarrow0$ as $x\rightarrow\infty$. Consider the equation 
	\begin{align}
		Ax^r\exp(-ax^{1/N})(1+g(x))=z,\label{eq242}
	\end{align}
	and notice that there is a unique solution $x$ provided $z$ is sufficiently small. Then, as $z\rightarrow0$,
	\begin{align}
		x\sim\frac{1}{a^N}\big(-\ln(z)\big)^N.\label{eq243}
	\end{align}
\end{lemma}

\begin{proof} From equation ($\ref{eq242}$) we get that $x=
		(\ln(w)+r\ln(x)+\ln(h(x)))^N/a^N$, where $w=A/z$ and $h(x)=1+g(x)$.
	From this equation it is clear that $x=O((\ln(w))^N)$ as $w\rightarrow\infty$. From this, we deduce that $\ln(x)=O(\ln(\ln(w))$ as $w\rightarrow\infty$, and that $\ln(h(x))\rightarrow0$ as $w\rightarrow\infty$ (because $\ln(1+u)\rightarrow0$ as $u\rightarrow0$). 
    Thus, $x\sim(\ln(w))^N/{a^N}\sim(-\ln(z))^N/{a^N}$ as $z\rightarrow0$.
\end{proof}   

\noindent{\emph{Proof of Proposition \ref{propq}.}} We derive the limiting form (\ref{Q1}); the derivation of (\ref{Q2}) is similar and thus omitted. Since $Q(p)$ satisfies $\bar{F}_Z(Q(p))=1-p$, it can be seen upon applying the limiting form (\ref{lim3}) that $Q(p)$ solves an equation of the form (\ref{eq242}) with $z=1-p$ and $a=N\omega_+^{1/N}$. The limiting form (\ref{Q1}) thus follows upon applying (\ref{eq243}) with these values of $z$ and $a$.
\qed

\section{Special cases}\label{sec3}

\subsection{Product of $N$ independent asymmetric Laplace random variables}\label{sec3.1}

In section, we apply the general formulas of Section \ref{sec2} to obtain closed-form formulas for the PDF,  CDF and characteristic function of the product of $N$ independent asymmetric Laplace random variables. We recall that the $\mathrm{VG}(1/2,\alpha,\beta,0)$ distribution corresponds to the asymmetric Laplace distribution (with zero location parameter), denoted by $\mathrm{AL}(\alpha,\beta)$, with density
$f(x)=(2\alpha)^{-1}(\alpha^2-\beta^2)\mathrm{e}^{\beta x-\alpha|x|}$, $x\in\mathbb{R}$ (see \cite{kkp01}).
Further setting $\beta=0$ yields the classical Laplace distribution, denoted by $\mathrm{Laplace}(\alpha)$, with density $f(x)=(\alpha/2)\mathrm{e}^{-\alpha |x|}$, $x\in\mathbb{R}$.

\begin{corollary}\label{cor3.1} Let $X_i\sim\mathrm{AL}(\alpha_i,\beta_i)$, $i=1,\ldots,N$ be independent asymmetric Laplace random variables, with $0\leq|\beta_i|<\alpha_i$, $i=1,\ldots,N$. Let $Z=\prod_{i=1}^{N}X_i$. Then

\vspace{2mm}

\noindent	(i) For $z\in\mathbb{R}$,
	\begin{align*}
		f_Z(z)= & \frac{1}{2^N\xi}  \bigg\{\prod_{i=1}^{N}\gamma_i^2\bigg\}
		\sum_{\sigma\in A(z)}G^{N,0}_{0,N} \bigg(\omega_\sigma|z| \,\bigg|\, \begin{aligned} -\atop  0,\ldots,0  \end{aligned}   \bigg).
	\end{align*}
\noindent (ii) We have
\begin{align*}
F_Z(z)&=1-\frac{1}{2^N\xi}  \bigg\{\prod_{i=1}^{N}\gamma_i^2\bigg\} \sum_{\sigma\in S_N^+}\frac{1}{\omega_\sigma}\bigg\{1-G^{N,1}_{1,N+1} \bigg( \omega_\sigma z\,\bigg|\, \begin{aligned} &1 \\ 1, \ldots&, 1,0 \end{aligned}   \bigg)\bigg\}, \quad z>0, \\
F_Z(z)&=\frac{1}{2^N\xi}  \bigg\{\prod_{i=1}^{N}\gamma_i^2\bigg\}\sum_{\sigma\in S_N^-}\frac{1}{\omega_\sigma}\bigg\{1-G^{N,1}_{1,N+1} \bigg( -\omega_\sigma z\,\bigg|\, \begin{aligned} &1 \\ 1, \ldots&, 1,0 \end{aligned}   \bigg)\bigg\}, \quad z<0.
\end{align*}
\noindent	(iii) For $t\in\mathbb{R}$,
	\begin{align}
		\varphi_Z(t)&=\frac{\mathrm{i}}{2^N\xi t}\bigg\{\prod_{j=1}^{N}\gamma_j^2\bigg\}\bigg\{\sum_{\sigma \in S_N^+}G^{N,1}_{1,N} \bigg(\frac{\mathrm{i}\omega_\sigma}{t}\, \bigg|\, \begin{aligned} 0 \atop 0, \ldots,0  \end{aligned}   \bigg)-\sum_{\sigma \in S_N^-}G^{N,1}_{1,N} \bigg(-\frac{\mathrm{i}\omega_\sigma}{t}\, \bigg|\, \begin{aligned} 0 \atop 0, \ldots,0  \end{aligned}   \bigg)\bigg\}. \label{charlap}
	\end{align}
\end{corollary}

\begin{corollary}\label{cor3.2}
	Let	$X_i\sim\mathrm{Laplace}(\alpha_i)$, $i=1,\ldots,N$, be independent Laplace random variables. Let $Z=\prod_{i=1}^{N}X_i$. Then the PDF, CDF and characteristic function of $Z$ are given by
	\begin{align*}
			f_Z(z)&= \frac{\xi}{2}G^{N,0}_{0,N} \bigg(\xi|z|\, \bigg|\, \begin{aligned} -\atop  0,\ldots,0  \end{aligned}   \bigg),\quad z\in\mathbb{R}, \\
		F_Z(z)&=\frac{1}{2}+\frac{\mathrm{sgn}(z)}{2}G^{N,1}_{1,N+1}\bigg( \xi|z|\, \bigg|\,{ 1\atop 1,\ldots,1,0}\bigg),\quad z\in\mathbb{R},\\
		\varphi_Z(t)&=\frac{\xi}{2^{N-1}\pi^{(N-1)/2}}|t|^{-1}G^{2N-1,1}_{1,2N-1} \bigg( \frac{\xi^2}{4^{N-1}t^2} \,\bigg|\, \begin{aligned} \frac{1}{2}\atop  0, \ldots,0, \,\frac{1}{2},\ldots, \frac{1}{2}  \end{aligned}   \bigg),\quad t\in\mathbb{R}.
	\end{align*}
\end{corollary}

\begin{remark} In the case $N=2$, we can apply the reduction formula (\ref{wmfor}) to formula (\ref{charlap}) to obtain the following expression for the characteristic function of $Z=X_1X_2$, where $X_1\sim\mathrm{AL}(\alpha_1,\beta_1)$ and $X_2\sim\mathrm{AL}(\alpha_2,\beta_2)$ are independent.
For $t\in\mathbb{R}$,
\begin{align}\varphi_Z(t)&=\frac{\gamma_1^2\gamma_2^2}{4\alpha_1\alpha_2}\bigg(\frac{\mathrm{i}}{t}\bigg)^{1/2}\bigg\{\frac{\mathrm{e}^{\mathrm{i}\lambda_1^-\lambda_2^-/(2t)}}{(\lambda_1^-\lambda_2^-)^{1/2}}W_{-\frac{1}{2},0}\bigg(\frac{\mathrm{i}\lambda_1^-\lambda_2^-}{t}\bigg)+\frac{\mathrm{e}^{\mathrm{i}\lambda_1^+\lambda_2^+/(2t)}}{(\lambda_1^+\lambda_2^+)^{1/2}}W_{-\frac{1}{2},0}\bigg(\frac{\mathrm{i}\lambda_1^+\lambda_2^+}{t}\bigg)\nonumber\\
&\quad+\frac{\mathrm{e}^{-\mathrm{i}\lambda_1^-\lambda_2^+/(2t)}}{(\lambda_1^-\lambda_2^+)^{1/2}}W_{-\frac{1}{2},0}\bigg(-\frac{\mathrm{i}\lambda_1^-\lambda_2^+}{t}\bigg)+\frac{\mathrm{e}^{-\mathrm{i}\lambda_1^+\lambda_2^-/(2t)}}{(\lambda_1^+\lambda_2^-)^{1/2}}W_{-\frac{1}{2},0}\bigg(-\frac{\mathrm{i}\lambda_1^+\lambda_2^-}{t}\bigg)\bigg\}, \nonumber
\end{align} 
where $W_{\kappa,\mu}(x)$ is a Whittaker function (see \cite[equation 13.14.3]{olver} for a definition) This formula corrects the formula of part (iii) of Corollary 3.1 of \cite{gl24} in which the denominator in the  argument of the Whittaker functions was erroneously given by $2t$ rather than by $t$.
\end{remark}

\noindent{\emph{Proof of Corollary \ref{cor3.1}.}} For parts (i) and (iii), we simply set $m_1=\cdots=m_N=1/2$ in (\ref{redf}) and (\ref{Cfhalf}), respectively. The proof of part (ii) is more involved. We prove this part for $z>0$; the case $z<0$ is similar and thus omitted. For $z>0$, we integrate the PDF from part (i) to get
\begin{align}
F_Z(z)&=1-\frac{1}{2^N\xi}  \bigg\{\prod_{i=1}^{N}\gamma_i^2\bigg\}
		\sum_{\sigma\in A(z)}\int_z^\infty G^{N,0}_{0,N} \bigg(\omega_\sigma y \,\bigg|\, \begin{aligned} -\atop  0,\ldots,0  \end{aligned}   \bigg)  \,\mathrm{d}y \nonumber\\
&=1-\frac{1}{2^N\xi}  \bigg\{\prod_{i=1}^{N}\gamma_i^2\bigg\}
		\sum_{\sigma\in A(z)}\frac{1}{\omega_\sigma}\int_{\omega_\sigma z}^\infty G^{N,0}_{0,N} \bigg(u \,\bigg|\, \begin{aligned} -\atop  0,\ldots,0  \end{aligned}   \bigg)  \,\mathrm{d}u \nonumber \\
&=1-\frac{1}{2^N\xi}  \bigg\{\prod_{i=1}^{N}\gamma_i^2\bigg\}
		\sum_{\sigma\in A(z)}\frac{1}{\omega_\sigma} \bigg[G^{N,1}_{1,N+1} \bigg(u \,\bigg|\, \begin{aligned} 1\atop  1,\ldots,1,0  \end{aligned}   \bigg) \bigg]_{\omega_\sigma z}^\infty,     \label{int789} 
\end{align}
where we used the  integral formula (\ref{mint1}) in the last step. Finally, we calculate the limit
\begin{align}
\lim_{u\rightarrow \infty}G^{N,1}_{1,N+1} \bigg(u \,\bigg|\, \begin{aligned} 1\atop  1,\ldots,1,0  \end{aligned}   \bigg)=\lim_{v\rightarrow 0}G^{1,N}_{N+1,1} \bigg(v \,\bigg|\, \begin{aligned} 0,\ldots,0,1\atop  0  \end{aligned}   \bigg)=1, \label{lim789}    
\end{align}
where we used the identity (\ref{ginv}) (and set $v=1/u$) to obtain the equality, and the final limit is easily computed by writing the Meijer $G$-function as a contour integral using (\ref{mdef}) and then computing the residues as we did in the proof of Proposition \ref{vg0}; we omit the details. Applying the limit (\ref{lim789}) to equation (\ref{int789}) now yields the desired formula for $F_{Z}(z)$ for $z>0$.
\qed

\vspace{3mm}
\noindent{\emph{Proof of Corollary \ref{cor3.2}.}} For the PDF, set $\beta_1=\cdots=\beta_N=0$ in part (i) of Corollary \ref{cor3.1}. For the CDF and characteristic function, set $m_1=\cdots=m_N=1/2$ in formulas (\ref{p2.73}) and (\ref{char1}), respectively.
\qed

\subsection{Product of independent zero mean normal and Laplace random variables}\label{sec3.2}

Recall 
that the $\mathrm{VG}(1/2,\alpha,0,0)$ distribution corresponds to a Laplace distribution with density $f(x)=(\alpha/2)\mathrm{e}^{-\alpha|x|}$, $x\in\mathbb{R}$. The product of independent zero mean normal random variables is also VG distributed:
Let $X_1\sim N(0,\sigma_1^2)$ and $X_2\sim N(0,\sigma_2^2)$ be independent. Then, by \cite[Proposition 1.2]{gaunt vg}, $X_1X_2\sim \mathrm{VG}(0,1/(\sigma_1\sigma_2),0,0)$.

\begin{corollary} Let $X_1,\ldots,X_{2M}$ and $Y_1,\ldots,Y_N$ be mutually independent random variables with $X_i\sim N(0,\sigma_i^2)$, $i=1,\ldots,2M$, and $Y_j\sim\mathrm{Laplace}(\alpha_j)$, $j=1,\ldots,N$. Let $Z=(\prod_{i=1}^{2M}X_i)(\prod_{j=1}^N Y_j)$. Denote $\nu=(\prod_{i=1}^{2M}\sigma_i^{-1})(\prod_{j=1}^N\alpha_j)$.  Also, let $\mathbf{0}$ and $\mathbf{0}'$ be the zero vectors of dimensions $2M+N$ and $2M+N-1$, respectively, and let $\boldsymbol{\frac{1}{2}}$ be a vector of dimension $N$ with all entries set to $1/2$. 
Then the PDF, CDF and characteristic function of $Z$ are given by
\begin{align*}
f_Z(z)&=\frac{\nu}{2^{M+N} \pi^{M+N/2}} G^{2(M+N),0}_{0,2(M+N)} \bigg( \frac{\nu^{2}z^2}{4^{M+N}}\, \bigg|\, \begin{aligned} - \atop  \mathbf{0}, \boldsymbol{\frac{1}{2}}  \end{aligned}   \bigg),\quad z\in\mathbb{R},   \\
F_Z(z)&=\frac{1}{2}+ \frac{\nu z }{  2^{M+N+1} \pi^{M+N/2}  } G^{2(M+N),1}_{1,2(M+N)+1}\bigg( \frac{\nu^2z^2}{4^{M+N}} \,\bigg|\,{ \frac{1}{2} \atop \mathbf{0}, \boldsymbol{\frac{1}{2}},-\frac{1}{2}}\bigg),\quad z\in\mathbb{R},\\
\varphi_Z(t)&=\frac{\nu |t|^{-1}}{2^{M+N-1}\pi^{M+(N-1)/2}}G^{2(M+N)-1,1}_{1,2(M+N)-1} \bigg( \frac{\nu^2}{4^{M+N-1}t^2} \,\bigg|\, \begin{aligned} \tfrac{1}{2}\atop  \mathbf{0}', \boldsymbol{\frac{1}{2}}  \end{aligned}   \bigg),\quad t\in\mathbb{R}.    
\end{align*}
\end{corollary}

\begin{proof} To obtain the formulas for the PDF, CDF and characteristic function, we apply formulas (\ref{simp}), (\ref{cdf1}) and (\ref{char1}), respectively, with
$m_1=\cdots=m_M=0$, $m_{M+1}=\cdots=m_{M+N}=1/2$ and $\xi=\nu$.
\end{proof}

\subsection{Product of correlated zero mean normal random variables}\label{sec3.3}

Consider a $2N$-dimensional multivariate normal random vector $(X_1,\ldots,X_{2N})^\intercal$ with zero mean vector and block diagonal covariance matrix $\Sigma=\mathrm{diag}(\Sigma_1,\Sigma_2,\ldots,\Sigma_N)$, where 
$\Sigma_i$ is a $2\times 2$ matrix given by
\begin{equation*}
\Sigma_i=\left[\begin{aligned} &\quad\,\sigma_{2i-1}^2  & \rho_i \sigma_{2i-1} \sigma_{2i} 
\\ &\rho_i \sigma_{2i-1} \sigma_{2i}   &\sigma_{2i}^2 \quad\,\,\,
\end{aligned} \right], \quad i=1,\ldots,N,    
\end{equation*}
and $\rho_1,\ldots,\rho_N\in(-1,1)$ are the correlation coefficients $\rho_i=\mathrm{Cov}(X_{2i-1},X_{2i})/(\sigma_{2i-1} \sigma_{2i})$. In this section, we apply Theorem \ref{thm1} to write down an exact formula for the PDF of the product $Z=\prod_{i=1}^{2N}X_i$. To this end, we note that the product $X_{2i-1}X_{2i}$ has a VG distribution:
\begin{equation}\label{vgrep}X_{2i-1}X_{2i}\sim\mathrm{VG}\bigg(0,\frac{1}{\sigma_{2i-1}\sigma_{2i}(1-\rho_i^2)},\frac{\rho_i}{\sigma_{2i-1}\sigma_{2i}(1-\rho_i^2)},0\bigg), \quad i=1,\ldots,N
\end{equation}
(see \cite[Theorem 1]{gaunt prod}). We remark that starting with the work of \cite{craig,wb32}, the distribution of the product of two correlated zero mean normal random variables 
has received considerable attention in the statistics literature (for an overview see \cite{gaunt22,np16}). On combining formula (\ref{eq:1}) for the VG product PDF with (\ref{vgrep}), we obtain the following formula for the PDF of $Z=\prod_{i=1}^{2N}X_i$.

\begin{corollary}Let the above notations hold, and set $s=\prod_{i=1}^{2N}\sigma_i$ and $\tau=\prod_{i=1}^N(1-\rho_i^2)$. Then
\begin{align}
		f_Z(z) & = \frac{1}{2^{2N-1}\pi^{N}s\sqrt{\tau}}
		\sum^\infty_{j_1=0} \ldots \sum^\infty_{j_N=0} \frac{(2\rho_{1})^{j_1}}{ j_1! } \ldots \frac{(2\rho_N)^{j_N} }{j_N!} a_{j_1,\ldots,j_N}(z) \nonumber
		\\ & \quad \times G^{2N,0}_{0,2N} \bigg( \frac{z^{2}}{4^Ns^2\tau^2}\, \bigg|\, { - \atop   \tfrac{j_1}{2},\ldots, \tfrac{j_N}{2}, \tfrac{j_1}{2}, \ldots, \tfrac{j_N}{2} }   \bigg), \quad z\in\mathbb{R}. \label{9pdf}
	\end{align}
\end{corollary}


\begin{remark} Note that the PDF (\ref{9pdf}) simplifies to a single Meijer $G$-function if  $\rho_1=\cdots=\rho_N=0$, in which case we recover the formula of \cite{product normal} for the product of $2N$ independent zero mean normal random variables. Similarly, if $N-k$ of the correlation coefficients are equal to zero then the PDF (\ref{eq:1}) reduces to a $k$-fold infinite series.

There is interest in developing distributional theory for the product of three or more correlated zero mean normal random variables (see the review for an overview \cite{o25}); however, to date few explicit formulas are known for these distributions. Formula (\ref{9pdf}) is one of the first such explicit formulas, treating the case of a block diagonal covariance matrix, and we hope the result inspires further research into the study of the exact distribution of the product of more than two correlated zero mean normal random variables.
\end{remark}

\appendix
\section{The Meijer $G$-function}\label{appa}
Here, we define the Meijer $G$-function, 
and present some of its relevant basic properties, all of which can be found in
\cite{
luke,olver}. 




The \emph{Meijer $G$-function} is defined, for $x\in\mathbb{R}$, by the contour integral
\begin{equation}\label{mdef}G^{m,n}_{p,q}\bigg(x \, \bigg|\, {a_1,\ldots, a_p \atop b_1,\ldots,b_q} \bigg)=\frac{1}{2\pi \mathrm{i}}\int_L\frac{\prod_{j=1}^m\Gamma(b_j-s)\prod_{j=1}^n\Gamma(1-a_j+s)}{\prod_{j=n+1}^p\Gamma(a_j-s)\prod_{j=m+1}^q\Gamma(1-b_j+s)}x^s\,\mathrm{d}s,
\end{equation}
where the path $L$ separates the poles of the gamma functions $\Gamma(b_j-s)$ from the poles of the gamma functions $\Gamma(1-a_j+s)$, and we employ the convention that the empty product is $1$.

The Meijer $G$-function satisfies the identities
\begin{align}\label{lukeformula}G_{p,q}^{m,n}\bigg(x \; \bigg| \;{a_1,\ldots,a_{p-1},b_1 \atop b_1,\ldots,b_q}\bigg)&=G_{p-1,q-1}^{m-1,n}\bigg(x \; \bigg| \;{a_1,\ldots,a_{p-1} \atop b_2,\ldots,b_q}\bigg), \quad m,p,q\geq 1, \\
\label{meijergidentity}x^\alpha G_{p,q}^{m,n}\bigg(x \; \bigg| \;{a_1,\ldots,a_p \atop b_1,\ldots,b_q}\bigg)&=G_{p,q}^{m,n}\bigg(x \, \bigg| \,{a_1+\alpha,\ldots,a_p+\alpha \atop b_1+\alpha,\ldots,b_q+\alpha}\bigg), \\
G^{m,n}_{p,q}\bigg(x \; \bigg|\; {a_1,\ldots, a_p \atop b_1,\ldots,b_q} \bigg)&=G^{n,m}_{q,p}\bigg(x^{-1} \; \bigg|\; {1-b_1,\ldots, 1-b_q \atop 1-a_1,\ldots,1-a_p} \bigg). \label{ginv}
\end{align}
The Whittaker function arises as a special case of the Meijer $G$-function:
\begin{align}
\label{wmfor}G^{2,1}_{1,2} \bigg( x\, \bigg|\, { a \atop b,c } \bigg)&=\Gamma(b-a+1)\Gamma(c-a+1)x^{(b+c-1)/2}\mathrm{e}^{x/2}W_{\frac{1}{2}(2a-b-c-1),\frac{1}{2}(b-c)}(x).  
\end{align}
On combining equations 5.4(1) and 5.4(13) of \cite{luke}, we obtain the indefinite integral formula
\begin{equation}\label{mint}\int x^{\alpha-1}G_{p,q}^{m,n}\bigg(x \; \bigg| \;{a_1,\ldots,a_p \atop b_1,\ldots,b_q}\bigg)\,\mathrm{d}x=x^\alpha G_{p+1,q+1}^{m,n+1}\bigg(x \, \bigg| \,{1-\alpha,a_1,\ldots,a_p \atop b_1,\ldots,b_q,-\alpha}\bigg).
\end{equation}
Setting $\alpha=1$ in (\ref{mint}) and applying (\ref{meijergidentity}) yields 
\begin{equation}\label{mint1}\int G_{p,q}^{m,n}\bigg(x \; \bigg| \;{a_1,\ldots,a_p \atop b_1,\ldots,b_q}\bigg)\,\mathrm{d}x= G_{p+1,q+1}^{m,n+1}\bigg(x \, \bigg| \,{1,a_1+1,\ldots,a_p+1 \atop b_1+1,\ldots,b_m+1,0,b_{m+1}+1,\ldots, b_q+1}\bigg).
\end{equation}
The following limiting form is given in \cite[Section 5.7, Theorem 5]{luke}. For $x>0$,
\begin{equation}\label{asymg}G^{q,0}_{p,q}\bigg(x \; \bigg|\; {a_1,\ldots, a_p \atop b_1,\ldots,b_q} \bigg)\sim \frac{(2\pi)^{(\sigma-1)/2}}{\sigma^{1/2}}x^\theta \exp\big(-\sigma x^{1/\sigma}\big), \quad \text{as $x\rightarrow\infty$,}
\end{equation}
where $\sigma=q-p$ and $\theta=\sigma^{-1}\{(1-\sigma)/2+\sum_{i=1}^qb_i-\sum_{i=1}^pa_i\}$.
\end{document}